\documentclass[12pt]{article}

\usepackage{amsthm}
\usepackage{amssymb}
\usepackage{amsmath}
\usepackage{amsfonts}

\usepackage[paper=a4paper, top=2.5cm, bottom=2.5cm, left=2.5cm, right=2.5cm]{geometry}

\def \Sph{\mathbb{S}^{n-1}}
\def \RR {\mathbb R}

\def \EE {\mathbb E}

\def \PP {\mathbb P}

\def \vphi {\varphi}

\newtheorem{theorem}{Theorem}
\newtheorem{lemma}[theorem]{Lemma}
\newtheorem{proposition}[theorem]{Proposition}
\newtheorem{corollary}[theorem]{Corollary}
\theoremstyle{definition}

\theoremstyle{remark}
\newtheorem*{remark}{Remark}

\begin{document}

\title{Bounding the norm of a log-concave vector via thin-shell estimates}
\author{Ronen Eldan and Joseph Lehec}
\date{}
\maketitle
\begin{abstract}
Chaining techniques show that if
$X$ is an isotropic log-concave random 
vector in $\mathbb{R}^n$ and $\Gamma$
is a standard Gaussian vector then
\[
\mathbb E \Vert X \Vert \leq C n^{1/4} \mathbb E \Vert \Gamma \Vert 
\]
for any norm $\Vert \cdot \Vert$, where $C$ is a universal constant. 
Using a completely different argument we establish a similar inequality
relying on the \emph{thin-shell} constant
\[ 
\sigma_n = \sup \Bigl( \sqrt{ \mathrm{var} ( |X| ) }  ; \ X \text{ isotropic and log-concave on }\mathbb R^n \Bigr) .
\] 
In particular, we show that if the
thin-shell conjecture $\sigma_n = O(1)$  holds,
then $n^{1/4}$ can be replaced by $\log (n)$ in the inequality. 
 As a consequence, we obtain certain bounds for
 the mean-width, the dual mean-width and the isotropic
 constant of an isotropic convex body. 
 In particular, we give an alternative proof of the fact that a positive answer
 to the thin-shell conjecture implies a positive answer to the slicing problem, up to a logarithmic factor.
\end{abstract}

\section{Introduction}

Given a stochastic process $(X_t)_{t \in T}$, the question of obtaining bounds for the quantity
\[
\EE \bigl( \sup_{t \in T} X_t  \bigr)
\]
is a fundamental question in probability theory dating back to Kolmogorov, and the theory behind this type of question has applications in a variety of fields. 

The case that $(X_t)_{t \in T}$ is a Gaussian process is perhaps the most important one. It has been studied intensively over the past 50 years, and numerous bounds on the supremum in terms of the geometry of the set $T$ have been attained by Dudley, Fernique, Talagrand and many others.

The case of interest in this paper is a certain generalization of the Gaussian process. We consider the supremum of the process 
\[Ê
( X_t = \langle X, t \rangle )_{t\in T}
\]
where $X$ is a log-concave random vector in $\RR^n$ and $T \subset \RR^n$ is a compact set. 
Throughout the article $\langle x , y \rangle$ denotes the inner product of $x,y \in \RR^n$ and 
$\vert x \vert = \sqrt{\langle x,x\rangle}$ the Euclidean norm of $x$.
Our aim is to obtain an upper bound on this supremum in terms of the supremum of a corresponding Gaussian process $Y_t = \langle \Gamma, t \rangle$ where $\Gamma$ is a gaussian random vector having the same covariance structure as $X$.

Before we formulate the results, we begin with some notation. A probability density $\rho: \RR^n \rightarrow [0, \infty)$
is called {\it log-concave} if it takes the form $\rho = \exp(-H)$
for a convex function $H: \RR^n \rightarrow \RR \cup \{ + \infty \}$.
A probability measure is log-concave if it has a log-concave density and a random vector taking values in $\RR^n$ 
is said to be log-concave if its law is log-concave. Two canonical examples of log-concave measures are the uniform probability measure on a convex body and the Gaussian measure. It is a well-known fact that any log-concave probability density decays
exponentially at infinity, and thus has moments of all orders. A log-concave random vector $X$
is said to be \emph{isotropic} if its expectation and covariance matrix satisfy
\[
 \EE ( X ) = 0 , \quad \mathrm{cov} ( X ) = \mathrm{id} . 
\]
Let $\sigma_n$ be the so-called \emph{thin-shell} constant:
\begin{equation}
 \sigma_n = \sup_X \sqrt{  \mathrm{var} (  \vert X \vert  ) }  \label{eq_1356}
\end{equation}
where the supremum runs over all isotropic, log-concave random vectors $X$ in $\RR^n$. 
It is trivial that $\sigma_n \leq \sqrt{n}$ and it was proven initially by Klartag~\cite{K2} 
that in fact
\[
\sigma_n = o(\sqrt{n}) . 
\]
Shortly afterwards, Fleury-Gu\'edon-Paouris~\cite{FGP} gave an alternative proof of this fact. Several improvements on the bound have been established since then,
and the current best estimate is $\sigma_n = O ( n^{1/3} )$ due to Gu\'edon-Milman~\cite{GM}.  
The \emph{thin-shell} conjecture,
which asserts that the sequence $(\sigma_n)_{n\geq 1}$ is bounded, 
is still open. 
Another related constant is:
\begin{equation} \label{defkn}
\tau_n^2 = \sup_X \sup_{\theta \in \Sph} \sum_{i,j=1}^n \EE \bigl( X_i X_j \langle X, \theta \rangle \bigr)^2,
\end{equation}
where the supremum runs over all isotropic log-concave random vectors $X$ in $\RR^n$.  
Although it is not known whether $\tau_n = O ( \sigma_n)$, we have the following estimate, proven in \cite{E1}
\begin{equation}
\label{ksigma}
\tau_n^2 = O \Bigl(  \sum_{k=1}^n \frac{\sigma_k^2} k  \Bigr).
\end{equation}
The estimate $\sigma_n = O ( n^{1/3} )$ thus gives $\tau_n = O ( n^{1/3} )$, 
whereas the \emph{thin-shell} conjecture yields $\tau_n = O ( \sqrt{ \log n } )$. 
\\ \\
We denote by $\Gamma$ the standard Gaussian vector in $\RR^n$ (with identity covariance matrix). We are now ready to formulate our main theorem. 
\begin{theorem} \label{mainthm}
Let $X$ be an isotropic log-concave random vector in $\RR^n$ and let $\Vert \cdot \Vert$ be a norm. There is a universal constant $C$ such that 
\begin{equation}
\label{main-result1}
\EE \Vert X \Vert \leq C \sqrt{\log n} \, \tau_n\,  \EE \Vert \Gamma \Vert . 
\end{equation}
\end{theorem}
\begin{remark}
It is well-known that an isotropic random vector satisfies the following $\psi_2$ estimate
\[
\PP  \bigl ( \vert \langle X , \theta \rangle \vert \geq t ) \leq C e^{ - c t^2 / \sqrt{n} } , 
\quad \forall t \geq 0 , \, \forall \theta \in \Sph ,
\]
where $C,c$ are universal constants. 
Combining this with chaining methods developed by Dudley-Fernique-Talagrand 
(more precisely, using Theorem~1.2.6. and Theorem~2.1.1. of~\cite{T}), one gets the inequality
$$
\EE  \Vert X \Vert  \leq C' n^{1/4} \EE  \Vert \Gamma \Vert  ,
$$
we refer to~\cite{Bou} for more details. 
This means that using the current best-known bound for the thin-shell constant: $\sigma_n = O ( n^{1/3} )$, the above theorem does not give us anything new.
\\
On the other hand, under the thin-shell hypothesis we obtain using~\eqref{ksigma}
\[
\EE \Vert X \Vert \leq C  \log n \,   \EE \Vert \Gamma \Vert . 
\]
\end{remark}
As an application of Theorem \ref{mainthm}, we derive several bounds related to the mean width and dual mean width of isotropic convex bodies and to the so-called {\em hyperplane conjecture}. We begin with a few definitions.
A convex body $K \subset \RR^n$ is a compact convex set whose interior contains the origin. For $x \in \RR^n$, we define
\[
\Vert x \Vert_K = \inf \{\lambda; ~ x \in \lambda K \}
\]
to be the gauge associated to $K$ (it is a norm if $K$ is symmetric about $0$). 
The polar body of $K$ is denoted by 
\[
K^\circ = \{ y \in \RR^n ; \ \langle x , y \rangle \leq 1 , \, \forall x \in K \} . 
\]
Next we define
\[
\begin{split}
M(K) & = \int_{\Sph} \Vert x \Vert_K \, \sigma ( dx) , \\
M^*(K) & = \int_{\Sph} \Vert x \Vert_{K^\circ} \, \sigma (dx) ,
\end{split}
\]
where $\sigma$ is the Haar measure on the sphere, normalized to be a probability measure. 
These two parameters play an important r\^{o}le in the asymptotic theory of convex bodies. \\
A convex body $K$ is said to be isotropic if a random vector uniform on $K$ is isotropic. 
When $K$ is isotropic, the \emph{isotropic constant} of $K$ is then defined to be 
\[
L_K = \vert K \vert^{-1/n} ,
\]
where $\vert K \vert$ denotes the Lebesgue measure of $K$. 
More generally, the isotropic constant of an isotropic log-concave 
random vector is $L_X = f ( 0)^{1/n}$ where $f$ is the density of $X$. 
The {\em slicing} or {\em hyperplane} conjecture asserts that $L_K \leq C$ 
for some universal constant $C$. The current best estimate is 
$L_K \leq C n^{1/4}$ due to Klartag~\cite{K1}. 
We are ready to formulate our corollary:
\begin{corollary} \label{corslicing}
Let $K$ be an isotropic convex body. Then one has,
\begin{itemize}
\item[(i)] $M(K) \geq c / ( \sqrt{n \log n}\,  \tau_n)$,
\item[(ii)] $M^*(K) \geq c  \sqrt{n} / ( \sqrt{ \log n}  \, \tau_n )$,
\item[(iii)] $L_K \leq C \tau_n (\log n)^{3/2}$,
\end{itemize}
where $c,C>0$ are universal constants.
\end{corollary}
\begin{remark}
Part (iii) of the corollary is nothing new. Indeed, in \cite{EK2}, it is shown that $L_K \leq C \sigma_n$ for a universal constant $C>0$. Our proof uses different methods and could therefore shed some more light on this relation, which is the reason why we provide it.
\end{remark}
Using similar methods, we attain an alternative proof of the following correlation inequality proven initially by Harg\'{e} in \cite{harge}.
\begin{proposition}[Harg\'{e}]
\label{harge}
Let $X$ be a random vector on $\RR^n$.  
Assume that $\EE(X)=0$ and that $X$ is more log-concave than $\Gamma$,
i.e. the density of $X$ has the form 
\[ x\mapsto \exp \bigl( -V (x) - \frac 1 2  \vert x \vert^2  \bigr) 
\]
for some convex function $V \colon \RR^n \to  (-\infty , +\infty]$. 
Then for every convex function $\varphi \colon \RR^n \to \RR$
we have 
\[
\EE  \varphi (X)  \leq \EE \varphi ( \Gamma ) . 
\]
\end{proposition}
The structure of the paper is as follows: in section 2 we recall some properties of a stochastic process constructed in~\cite{E1}, which will serve as one of the central ingredients in the proof of Theorem~\ref{mainthm}, as well as establish some new facts about this process. In section 3 we prove the main theorem and Proposition~\ref{harge}. Finally, in section 4 we prove Corollary~\ref{corslicing}.
\\ \\
In this note, the letters $c,
\tilde{c}, c^{\prime}, C, \tilde{C}, C^{\prime}, C''$ will denote positive
universal constants, whose value is not necessarily the same in
different appearances. Further notation used throughout
the text: $\mathrm{id}$ will denote the identity $n \times n$ matrix. The Euclidean unit sphere is denoted by $\Sph = \{ x
\in \RR^n ; \vert x \vert = 1 \}$. The operator norm and the trace of a matrix $A$ are denoted by $\Vert A \Vert_{op}$
and $\mathrm{tr}(A)$, respectively. For two probability measures $\mu$, $\nu$ on $\RR^n$, we let 
$T_2 ( \mu , \nu)$ be their \emph{transportation cost} for the Euclidean distance squared:
$$
T_2(\mu, \nu) = \inf_\xi \int_{\RR^n \times \RR^n} \vert x - y \vert^2 \, \xi( dx, dy)
$$
where the infimum is taken over all measures $\xi$ on $\RR^{2n}$ whose marginals 
onto the first and last $n$ coordinates are the measures $\mu$ and $\nu$ respectively. 
Finally, given a continuous martingale $(X_t)_{t\geq 0}$, 
we denote by $[X]_t$ its quadratic variation. If $X$ is $\RR^n$ valued, then $[X]_t$ is a non-negative 
matrix whose $i,j$ coefficient is the quadratic covariation
of the  $i$-th and $j$-th coordinates of $X$ at time $t$. 
\\ \\ 
\emph{Acknowlegements}.
The authors wish to thank Bo'az Klartag for a fruitful discussion and Bernard Maurey for allowing them to use an unpublished result of his. 

\section{The stochastic construction}

We make use of the construction described in \cite{E1}. 
There it is shown that, given a probability measure $\mu$ 
having compact support and whose density with respect to the Lebesgue measure is $f$, 
and given a standard Brownian motion $(W_t)_{t\geq 0}$ on $\RR^n$;
there exists an adapted random process $(\mu_t )_{t \geq 0}$ taking 
values in the space of absolutely continuous probability measures
such that $\mu_0 = \mu$ and such that the density $f_t$ of $\mu_t$ satisfies
\begin{equation}
\label{diff-ft}
d f_t ( x)   =    f_t (x) \langle A_t^{-1/2} (x - a_t) , d W_t \rangle  , \quad \forall t \geq 0 , 
\end{equation}
for every $x\in\RR^n$, where
\[
\begin{split}
a_t & = \int_{\RR^n} x \, \mu_t ( dx ), \\
A_t & = \int_{\RR^n} (x - a_t) \otimes (x - a_t)  \, \mu_t( dx )
\end{split}
\]
are the barycenter and the covariance matrix of $\mu_t$, respectively. 

Let us give now the main properties of this process. Some of these properties have already been established in \cite{E1}, in this case we will only give the general idea of the proof. We refer the reader to \cite[Section 2,3]{E1} for complete proofs.
Firstly, for every test function $\phi$
the process 
\[
\Bigl( \int_{\RR^n} \phi \, d \mu_t \Bigr)_{t\geq 0}
\]
is a martingale. In particular
\begin{equation}
\label{average}
\EE \int_{\RR^n} \phi \, d \mu_t  = \int_{\RR^n} \phi \, d\mu , \quad \forall t \geq 0 .
\end{equation}
The It\^o differentials of $a_t$ and $A_t$ read
\begin{align} 
\label{diff-at}
d a_t & = A_t^{1/2} d W_t \\
\label{diff-At}
d A_t &= - A_t \, dt + \int_{\RR^n} (x-a_t)\otimes (x-a_t)  \langle A_t^{-1/2} (x-a_t) , d W_t \rangle \, \mu_t ( dx ) .  
\end{align}
It follows from the second equation that 
\[
\frac d {dt} \EE \mathrm{tr} ( A_t ) = - \EE \mathrm{tr} ( A_t ) . 
\]
Integrating this differential equation we obtain
\begin{equation}
\label{elem-At}
\EE \mathrm{tr} ( A_t ) = e^{- t } \mathrm{tr} ( A_0 )  , \quad t \geq 0 .  
\end{equation}
Combining this with~\eqref{diff-at} we obtain
\[
\EE \vert a_t \vert^2 = \vert a_0 \vert^2 + \int_0^t \EE \mathrm{tr} (A_s ) \, ds = \vert a_0 \vert^2 + ( 1- e^{-t} ) \mathrm{tr} ( A_0 ) . 
\]
The process $(a_t)_{t \geq 0}$ is thus a martingale bounded in $L^2$. By Doob's theorem, 
it converges almost surely and in $L^2$ to some random vector $a_\infty$. 
\begin{proposition} 
\label{samedist}
The random vector $a_\infty$ has law $\mu$.
\end{proposition}
\begin{proof}
Let $\phi,\psi$ be functions on $\RR^n$ satisfying
\begin{equation}
\label{kantorovich}
\phi( x) + \psi(y) \leq \vert x-y \vert^2 , \quad x,y\in \RR^n . 
\end{equation}
Then 
\[
\phi ( a_t )  + \int_{\RR^n} \psi (y) \, \mu_t ( dy) \leq \int_{\RR^n} \vert a_t - y \vert^2 \, d y =  \mathrm{tr} ( A_t ) . 
\]
Taking expectation and using~\eqref{average} and~\eqref{elem-At} we obtain 
\[
\int_{\RR^n} \phi \,  d \nu_t  + \int_{\RR^n} \psi \, d \mu  \leq  \mathrm{tr} ( A_0 ) e^{-t} ,
\]
where $\nu_t$ is the law of $a_t$. This holds for every pair of functions 
satisfying the constraint~\eqref{kantorovich}. By the Monge-Kantorovich duality
(see for instance~\cite[Theorem~5.10]{V})
we obtain
\[
T_2 (  \nu_t , \mu ) \leq e^{-t }  \mathrm{tr} ( A_0)
\]
where $T_2$ is the transport cost associated to the Euclidean 
distance squared, defined in the introduction. Thus $\nu_t \rightarrow \mu$ in the $T_2$
sense, which implies that $a_t \rightarrow \mu$ in law, hence the result. 
\end{proof}
%(iii) For all $t \geq 0$,
%\begin{equation} \label{brascamplieb}
%A_t \leq C || B_t^{-1} ||_{OP} id.
%\end{equation}
%(iii) There exists a universal constant $C>0$ such that the event
%\begin{equation} \label{deff}
%F := \left \{ ||A_t||_{OP} < C \tau_n^2 (\log n) e^{-c t}, ~~ \forall t > 0 \right \},
%\end{equation}
%satisfies
%\begin{equation} \label{onb1}
%\PP(F) ~~\geq~~ 1 - (n^{-10}).
%\end{equation}
%
Let us move on to properties of the operator norm of $A_t$. 
We shall use the following lemma
which follows for instance from a theorem of Brascamp-Lieb~\cite[Theorem~4.1.]{BL}. 
We provide an elementary proof using the Pr\'{e}kopa-Leindler inequality. 
\begin{lemma}
\label{bl}
Let $X$ be a random vector on $\RR^n$ whose density $\rho$ has the form
$$
\rho(x) = \exp \left ( - \frac{1}{2} \langle Bx, x \rangle - V(x) \right )
$$
where $B$ is a positive definite matrix, and  $V \colon \RR^n \to (-\infty + \infty] $ is a convex function.
Then one has,
\[
\mathrm{cov} ( X ) \leq B^{-1}.
\]
In other words, if a random vector $X$ is more log-concave than a Gaussian vector $Y$,
then $\mathrm{cov} ( X ) \leq \mathrm{cov} ( Y )$.
\end{lemma}
\begin{proof}
There is no loss of generality assuming that $B = \mathrm{id}$ (replace $X$ by $B^{1/2} X$ otherwise). 
Let 
\[
\Lambda \colon  x \mapsto \log \EE ( e^{\langle x,X \rangle } )  .
\]
Since log-concave vectors have exponential moment $\Lambda$ is $\mathcal C^{\infty}$
in a neighborhood of $0$ and it is easily seen that
\begin{equation}
\label{step-lemma-bl}
\nabla^2 \Lambda (0) = \mathrm{cov} ( X) .  
\end{equation}
Fix $a\in \RR^n$ and define
\[
\begin{split}
f \colon x &\mapsto \langle a , x \rangle - \frac 1 2 \vert x\vert^2 - V(x) , \\
g \colon y & \mapsto  -\langle a , y \rangle - \frac 1 2 \vert y \vert^2 - V(y)  ,\\
h \colon z& \mapsto  - \frac 1 2 \vert z \vert^2 - V(z)  . 
\end{split}
\]
Using the inequality 
\[
 \frac 12 \langle a , x -y  \rangle  - \frac 1 4 \vert x\vert^2   - \frac 1 4 \vert y \vert^2  \
 \leq \frac 12 \vert a \vert^2 - \frac 1 8 \vert x+y \vert^2 , 
 \]
 and the convexity of $V$ we obtain 
\[
\frac 12 f (x) + \frac 12 g(y) \leq  \frac 1 2 \vert a \vert^2  + h \bigl( \frac{x+y} 2 \bigr)  , \quad \forall x,y\in \RR^n .
\] 
Hence by Pr\'ekopa-Leindler
\[
\Bigl( \int_{\RR^n} e^{f(x)} \, dx \Bigl)^{1/2}
\Bigl( \int_{\RR^n} e^{g(y)} \, dy \Bigl)^{1/2} \leq 
e^{\vert a\vert^2 /2 } \, \int_{\RR^n} e^{h(z)} \, dz . 
\]
This can be rewritten as
\[
\frac 1 2 \Lambda ( a)  + \frac 1 2 \Lambda (-a) - \Lambda (0 ) \leq \frac 12 \vert a \vert^2 .
\]
Letting $a$ tend to $0$ we obtain $\langle \nabla^2 \Lambda (0) a , a\rangle \leq \vert a \vert^2$
which, together with~\eqref{step-lemma-bl}, yields the result. 
\end{proof}
Integrating~\eqref{diff-ft} shows that the density of the measure $\mu_t$ satisfies
\begin{equation}
\label{integ-ft}
f_t (x ) = f (x) \exp \Bigl( c_t  +  \langle b_t , x \rangle - \frac 1 2 \langle B_t x , x \rangle \Bigr) 
\end{equation}
where $c_t,b_t$ are some random processes, and 
\begin{equation}
\label{def-Bt}
B_t = \int_0^t A_s^{-1} \, ds .
\end{equation}
\begin{lemma}
\label{bl-app}
If the initial measure $\mu$
is more-log-concave than the
standard Gaussian measure, 
then almost surely
\[
\Vert  A_t \Vert_{op} \leq e^{-t}  , \quad \forall t \geq 0 . 
\]
%
%he operator norm of the matrix $A_t$ satisfies 
%the following differential inequality:
%\[
%\forall t \geq 0 , \quad \Vert A \Vert^{-1} \geq \int_0^t \Vert A_s \Vert^{-1} \, ds + \lambda . 
%\]
\end{lemma}
\begin{proof}
Let $\lambda_t$ be the lowest eigenvalue of $B_t$.  Define $Y$ to be the Gaussian random vector whose convariance matrix is
\[
\frac 1 { \lambda_t + 1 } \, \mathrm{id}  . 
\]
Then~\eqref{integ-ft} and the hypothesis show that the density of $\mu_t$ with respect to the law of $Y$ is 
log-concave. Therefore, by the previous lemma, the covariance matrix of $\mu_t$
satisfies
\[
A_t   \leq \frac 1 { \lambda_t + 1 } \, \mathrm{id} ,
\]
hence
\[
\Vert A_t  \Vert_{op} \leq \frac 1 { \lambda_t + 1} . 
\]
On the other hand, the equality~\eqref{def-Bt} yields
\[
\lambda_t \geq \int_0^t \Vert A_s\Vert_{op}^{-1} \, ds , 
\]
showing that 
\[
\int_0^t \Vert A_s  \Vert_{op}^{-1} \, ds +1 \leq \Vert A_t \Vert_{op}^{-1} .   
\]
Integrating this differential inequality yields the result. 
% we easily get
%\[
%\Vert A_t \Vert_{op} \leq e^{-t} ,
%\]
%which is the result (recall that $A_t$ is a symmetric positive matrix). 
\end{proof}
%%
%Combining this with a Gronvall type argument 
%we obtain easily the following corollary, which 
%we shall use in the last section. 
%%
%\begin{corollary}
%\label{striclty-convex}
%If $\lambda >0$ then
%\[
%\forall t \geq 0 , \quad \Vert A_t \Vert_{op} \leq  \frac{ e^{-t} }{ \lambda } . 
%\]
%\end{corollary}
%\begin{proof}
%Letting $H(t) = \int_0^t \Vert A_s \Vert^{-1} \, ds + \lambda$
%we obtain $H' \geq H$. Therefore
%\[
%\Vert A_t \Vert^{-1} \geq H(t) \geq e^t H(0) = e^t \, \lambda ,
%\]
%which proves the second part of the lemma. 
%\end{proof}
The following proposition will be crucial for the proof of our main theorem. Its proof is more involved than the proof of previous estimate, and we refer to~\cite[Section 3]{E1}.
\begin{proposition}
\label{main-tool}
If the initial measure $\mu$ is log concave then
\[
\EE \Vert A_t \Vert_{op} \leq C_0 \Vert A_0 \Vert_{op} \tau_n^2 \log ( n ) \, e^{-t} , \quad \forall t \geq 0 ,
\]
where $C_0$ is a universal constant. 
\end{proposition}

\section{Proof of the main theorem}

We start with an elementary lemma. 
%Let $\mu$ be an isotropic log-concave measure. We denote by $X_\mu$ the corresponding random vector. We begin with the following simple lemma.
\begin{lemma} \label{conclemma}
Let $X$ be a log-concave random vector in $\RR^n$ and let $\Vert \cdot \Vert$ be a norm. 
Then for any event $F$
\[
\EE\bigl( \Vert X\Vert ; \, F \bigr) \leq C_1 \sqrt{\PP( F)} \, \EE\bigl( \Vert X\Vert \bigr) ,
\]
where $C_1$ is a universal constant. 
In particular, if  $\PP(F) \leq ( 2 C_1)^{-2}$, one has
\begin{equation} \label{conceq}
\EE\bigl( \Vert X\Vert \bigr) \leq 2 \EE\bigl( \Vert X\Vert ; \, F^c \bigr) ,
\end{equation}
where $F^c$ is the complement of $F$. 
\end{lemma}
\begin{proof}
This is an easy consequence of Borell's lemma, which states as follows. 
There exist universal constants $C,c>0$ such that,
$$
\PP\Bigl( \Vert X\Vert > t \EE\bigl( \Vert X\Vert \bigr) \Bigr) \leq C e^{- c  t }.
$$
By Fubini's theorem and the Cauchy-Schwarz inequality
$$
\EE \bigl( \Vert X\Vert ; \, F \bigr) =   \int_0^\infty \PP\bigl(\Vert X\Vert > t, \, F \bigr) \, dt 
\leq \Bigl( \int_0^\infty \sqrt{\PP\bigl( \Vert X\Vert > t \bigr) } \, dt \Bigr) \times  \sqrt{\PP(F)}.
$$
Plugging in Borell's inequality yields the result, with constant $C_1 = 2C/ c$. 
%Equation (\ref{conceq}) is an immediate consequence by choosing $c = (4C)^{-2}$.
\end{proof}
The next ingredient we will need is the following proposition, 
which we learnt from B.Maurey (\cite{M}). The authors are not aware of any published similar result. 
\begin{proposition} \label{Maurey}
Let $(M_t)_{t\geq 0}$ be a continuous martingale
taking values in $\RR^n$. Assume that $M_0 =0$ 
and that the quadratic variation 
of $M$ satisfies
\[
\forall t >0 , \quad [M]_t \leq \mathrm{id} ,
\]
almost surely. Then $(M_t)_{t\geq 0}$ converges almost surely, and the limit satisfies the following inequality. 
Letting $\Gamma$ be a standard Gaussian vector, we
have for every convex function $\varphi \colon \RR^n \to \RR\cup\{+\infty\}$
\[
\EE \varphi ( M_\infty ) \leq  \EE  \varphi ( \Gamma ) .
\]
\end{proposition}
\begin{proof}
The hypothesis implies that $M$ is bounded in $L^2$,
hence convergent by Doob's theorem.
Let $X$ be a standard Gaussian vector on $\RR^n$
independent of $(M_t)_{t\geq 0}$. 
We claim that
\[
Y = M_\infty + ( \mathrm{id} -[M]_\infty )^{1/2} X 
\]
is also a standard Gaussian vector. 
Indeed, for a fixed $x\in \RR^n$ one has
\[
\begin{split}
\EE \left ( e^{i \langle x , Y\rangle } \mid (M_t)_{t\geq 0}  \right )
& = \exp \left ( i  \langle x , M_\infty \rangle +  \frac 1 2 \langle [M]_\infty x , x \rangle - \frac 1 2 \vert x\vert^2 \right ) \\
& = \exp \left ( i L_\infty +  \frac 1 2  [L]_\infty - \frac 1 2 \vert x\vert^2 \right ) , 
\end{split}
\]
where $L$ is the real martingale defined by $L_t = \langle M_t , x \rangle$. 
It\^o's formula shows that 
\[  D_t = \exp \left ( i L_t + \frac 1 2 [L]_t  \right ) \]
is a local martingale. On the other hand the hypothesis yields
\[
\vert D_t \vert =   \exp  \left ( \frac 1 2 \langle [M]_t x , x \rangle \right ) \leq \exp \left  ( \frac 12 \vert x \vert^2 \right  ) 
\]
almost surely. This shows that $(D_t)_{t\geq 0}$ is a bounded martingale; in particular 
\[ \EE ( D_\infty ) = \EE ( D_0 ) = 1 , \]
since $M_0 =0$. Therefore
\[
\EE \bigl( e^{i \langle x , Y\rangle }  \bigr) = e^{- \vert x\vert^2 /2} ,
\]
proving the claim. Similarly (just replace $X$ by $-X$)
\[
Z  =  M_\infty - ( \mathrm{id} -[M]_\infty )^{1/2} X 
\]
is also standard Gaussian vector. Now, given a convex function $\phi$, we have
\[
\EE \varphi ( M_\infty ) = \EE \vphi \left (  \frac {Y + Z} {2} \right )
\leq \frac 12  \EE \left ( \varphi ( Y ) + \varphi ( Z ) \right ) = \EE \varphi ( Y ) ,
 \]
which is the result. 
\end{proof}
We are now ready to prove the main theorem.
\begin{proof}[Proof of Theorem~\ref{mainthm}]
Let us prove that given a norm $\Vert \cdot \Vert$ and
a log-concave vector $X$ satisfying $\EE ( X ) = 0$ we have 
\begin{equation}
\label{main-result2}
\EE \Vert X \Vert  \leq C \tau_n (\log n)^{1/2} \,   \Vert \mathrm{cov} ( X ) \Vert_{op}^{1/2} \,
 \EE \Vert \Gamma \Vert , 
\end{equation}
for some universal constant $C$. 
If $X$ is assumed to be isotropic, then $\mathrm{cov} ( X ) = \mathrm{id}$
and we end up with the desired inequality~\eqref{main-result1}.
\\
Our first step is to reduce the proof to the case that $X$ has a compact support. Assume that \eqref{main-result2} holds for such vectors, and for $r >0$, let $Y_r$ be a random vector distributed
according to the conditional law of $X$ given the event $\{ \vert X \vert \leq r\}$. 
Then $Y_r$ is a compactly supported log-concave vector, and by our assumption,
\begin{equation}
\label{stepY}
\EE \Vert Y_r - \EE (Y_r) \Vert  \leq C \tau_n ( \log n )^{1/2}   \Vert \mathrm{cov} ( Y_r ) \Vert_{op}^{1/2} \, 
\EE \Vert \Gamma \Vert  .
\end{equation}
Besides, it is easily seen by dominated convergence that
\[
\begin{split}
\lim_{r \rightarrow +\infty}  \EE \Vert Y_r - \EE Y_r \Vert & = \EE \Vert X \Vert  ,  \\
\limsup_{r \rightarrow +\infty}  \Vert \mathrm{cov} ( Y_r ) \Vert_{op}  & \leq \Vert \mathrm{cov} (X)  \Vert_{op}.
\end{split}
\]
So letting $r$ tend to $+\infty$ in~\eqref{stepY} yields~\eqref{main-result2}. 
Therefore, we may continue the proof under the assumption that $X$ is compactly supported.

We use the stochastic process $(\mu_t)_{t\geq 0}$ defined in the beginning of the previous section, with the starting law $\mu$ being the law of $X$.  \\
Let $T$ be the following stopping time:
\[
T = \inf \Bigl( t \geq 0 , \, \int_0^t A_s \, ds > C^2  \tau_n^2  \log n \, \Vert A_0 \Vert_{op}  \Bigr) , 
\]
where $C$ is a positive constant to be fixed later 
and with the usual convention that $\inf (\emptyset ) = + \infty$. 
Define the stopped process $a^T$ by
\[
( a^T )_t =  a_{ \min ( t, T ) }  . 
\]
By the optional stopping theorem, this process is also a martingale and by definition of $T$
its quadratic variation satisfies
\[
[ a^T ]_t \leq  C^2  \tau_n^2  \log n \, \Vert A_0 \Vert_{op}   , \quad \forall t \geq 0  . 
\]
Also $(a^T)_0 = a_0 = \EE ( X) = 0$. Applying Proposition \ref{Maurey} 
%to the process $a^T /( \tau_n \sqrt{C \log(n)} )$, 
we get
\begin{equation}
\label{step-aT}
\EE  \Vert a_T \Vert  = \EE \Vert (a^T)_\infty \Vert  \leq  C \tau_n (\log n)^{1/2} \,   \Vert A_0 \Vert_{op}^{1/2} \, \EE  \Vert \Gamma \Vert .
\end{equation}
On the other hand, using Proposition~\ref{main-tool} and Markov inequality we get
\[
\PP ( T < +\infty ) =  \PP \left ( \int_0^\infty \Vert A_s \Vert_{op} \, ds >  C^2  \tau_n^2  \log n \, \Vert A_0 \Vert_{op}  \right ) 
\leq \frac{ C_0 }{ C^2 } .  
\]
So $\PP ( T < + \infty )$ can be rendered arbitrarily small by choosing $C$ large enough. By Proposition~\ref{samedist} we have $a_\infty = X$ in law;
in particular $a_\infty$ is log-concave. If $\PP ( T <+\infty )$ is small 
enough, we get using Lemma~\ref{conclemma}
\[
\begin{split}
\EE \Vert X \Vert = \EE \Vert a_\infty \Vert 
& \leq  2 \EE \bigl(  \Vert a_\infty \Vert ; \, T =  \infty \bigr) \\
& = 2 \EE \bigl(  \Vert a_T \Vert ; \, T =  \infty \bigr)   \leq 2  \EE   \Vert a_T \Vert  . 
\end{split}
\]
Combining this with~\eqref{step-aT} and recalling that $A_0 = \mathrm{cov} ( X )$ we obtain
the result~\eqref{main-result2}.
\end{proof}
The proof of Proposition~\ref{harge} follows the same lines. The main difference is that
Proposition~\ref{bl-app} is used in lieu of Proposition~\ref{main-tool}.
\begin{proof}[Proof of Proposition~\ref{harge}]
Let $Y_r$ b a random vector distributed 
according to the conditional law of $X$ given $\vert X \vert \leq r$. 
Then $Y_r$ is also more log-concave than $\Gamma$ and 
\[
\EE \varphi ( Y_r ) \rightarrow \EE \varphi ( X ) 
\]
as $r\rightarrow +\infty$. So again we can assume that $X$ is compactly supported,
and consider the process $(\mu_t)_{t\geq 0}$ starting from the law of $X$. 
\\
By Lemma~\ref{bl-app}, the process $(a_t)_{t\geq 0}$ is a martingale whose quadratic variation satisfies
\[
[a]_t = \int_0^t A_s \, ds \leq \mathrm{id}, \quad \forall t \geq 0 ,
\]
almost surely. Since again $a_0 = \EE ( X) = 0$, Proposition~\ref{Maurey} yields the result. 
\end{proof}
\begin{remark}
This proof is essentially due to Maurey; although his (unpublished) argument relied on a different stochastic construction. 
\end{remark}
\section{Application to Mean Width and to the Isotropic Constant}
In this section, we prove Corollary~\ref{corslicing}.

Let $\Gamma$ be a standard Gaussian vector in $\RR^n$ and let $\Theta$ be a point uniformly distributed in $\Sph$.
Integration in polar coordinates shows that for any norm $\Vert \cdot \Vert$, 
\[
\EE \Vert \Gamma \Vert = c_n \EE \Vert \Theta \Vert  ,
\]
where 
\[
c_n = \EE \vert \Gamma \vert = \sqrt{n} + O ( 1) ,
\]
since $\Gamma$ has the thin-shell property. 
Theorem~\ref{mainthm} can thus be restated as follows. 
If $X$ is an isotropic log-concave random vector and $K$ is a convex body 
containing $0$ in its interior then
\begin{equation}
\label{main-result3}
\EE \Vert X \Vert_K \leq C \sqrt{n \log n} \, \tau_n\,  M ( K ) . 
\end{equation}
%the thin-shell concentration of the Gaussian density, it follows that there exists constants $c_n \to 1$ such that the parameters $M(K), M^*(K)$ satisfy,
%\begin{equation} \label{eqmm}
%\sqrt n M(K) = c_n \EE  \Vert \Gamma \Vert_K , ~~ \sqrt n M^{*}(K) = c_n \EE \Vert \Gamma \Vert_{K^{\circ}} .
%\end{equation}
Now let $K$ be an isotropic convex body and let $X$ be a random vector uniform on $K$. 
Then $\PP ( \Vert X \Vert_K \leq 1/2 ) = 1/ 2^n \leq 1/2$, so that by Markov inequality
%  \vert 
%Next, we notice that $\vert K /2  \vert < \frac{1}{2} Vol(K)$ and that for every $x \in K \setminus (K/2)$, one has $\Vert x \Vert_K \geq \frac{1}{2}$, which implies that
\[
\EE \Vert X \Vert_K \geq \frac{1}{2} \PP \left ( \Vert X \Vert_K \geq \frac 1 2 \right ) \geq \frac{1}{4}.
\]
Inequality~\eqref{main-result3} becomes
\[
M (K) \geq \frac c { \sqrt{n \log n } \,  \tau_n } ,  
\]
proving~\textit{(i)}.
\\
Since $X \in K$ almost surely, we have $\Vert X \Vert_{K^\circ} \geq \vert X \vert^2$, hence
\[
\EE \Vert X \Vert_{K^\circ} \geq \EE  \vert X \vert^2 = n . 
\]
Applying~\eqref{main-result3} to $K^\circ$ thus gives
\[
M^* ( K ) \geq \frac{c \sqrt{n}}{ \sqrt{\log n} \, \tau_n } , 
\]
which is~\textit{(ii)}.
\\
In~\cite{Bou}, Bourgain combined the inequality
\begin{equation}
\label{bourgain}
\EE \Vert X \Vert \leq C n^{1/4} \EE \Vert \Gamma \Vert 
\end{equation}
with a theorem of Pisier to get the estimate
\[
L_K \leq C n^{1/4} \log n . 
\]
Part~\textit{(iii)} of the corollary is obtained 
along the same lines, replacing~\eqref{bourgain}
by our main theorem. We sketch the argument for
completeness. 
\\
Recall that $K$ is assumed to be isotropic
and that $X$ is uniform on $K$. Let $T$ be
a positive linear map of determinant $1$. 
Then by the arithmetic-geometric inequality
\[
\EE \Vert X \Vert_{( T K )^\circ} \geq \EE \langle X , TX \rangle = \mathrm{tr} ( T ) \geq n .
\]
Applying~\eqref{main-result3} to $(TK)^\circ$ we get
\begin{equation} \label{mstareq}
M^* ( TK ) \geq  \frac{ c\sqrt{n} } { \sqrt{ \log n} \,   \tau_n } .
\end{equation}
Now we claim that given a convex body $K$
containing $0$ in its in interior, there exists 
a positive linear map $T$ of determinant $1$ such that
\begin{equation}
\label{step-pisier}
M^* ( TK ) \leq  C  \vert K \vert^{1/n} \sqrt{n}  \log n . 
\end{equation}
Taking this for granted and combining it with~\eqref{mstareq} we obtain
\[
\vert K \vert^{-1/n} \leq C'  ( \log n)^{3/2} \tau_n.
\]
which is part (iii) of the corollary. 
\\
It remains to prove the claim~\eqref{step-pisier}. Clearly 
\[
M^*(K) \leq M^* ( K-K ) ,
\]
and by the Rogers-Shephard inequality (see~\cite{RS})
\[
\vert K - K \vert \leq 4^n \vert K \vert . 
\]
This shows that it is enough to prove the claim when $K$ is symmetric about the 
origin. Now if $K$ is a symmetric convex body in $\RR^n$,
Pisier's Rademacher-projection estimate
together with a result of Figiel and Tomczak-Jaegermann 
(see e.g. \cite[Theorem~2.5 and Theorem~3.11]{P})
guarantee the existence of $T$ such that 
\[
M(TK) M^*(TK) \leq C  \log (n) ,
\]
where $C$ is a universal constant.
This, together with Urysohn's inequality
\[
M(TK) \geq \Bigl( \frac{ \vert B_2^n \vert }{\vert TK \vert}  \Bigr)^{1/n } \geq  \frac c { \sqrt{n} \vert K \vert^{1/n} } ,
\]
yields~\eqref{step-pisier}. 
%. In view of the above equation, a bound on $L_K$ may be attained in means of bounding the right hand side of (\ref{bourgain}) in terms of $M^*(T(K))$. But equation (\ref{thmk}) gives us
%$$
%\int_K \Vert x \Vert_{T(K)^{\circ}} dx \leq C L_K \tau_n \sqrt{\log n} \EE \Vert \Gamma \Vert_{T(K)^{\circ}} \leq C' L_K \sqrt{n \log n} \tau_n  M^*(T(K)),
%$$
%so combining the three last equations gives,
%$$
%n L_K^2 \leq C' L_K \tau_n n (\log n)^{3/2}
%$$
%and part (iii) of the corollary follows.
%

\end{document}